\documentclass{article}

\usepackage{amsfonts}
\usepackage{amssymb}
\usepackage{amsmath}
\usepackage{amsthm}

\usepackage{array}
\usepackage{graphicx}

\usepackage{authblk}    

\usepackage[utf8]{inputenc} 

\newtheorem{theorem}{Theorem}[section]

\newtheorem{lemma}[theorem]{Lemma}

\theoremstyle{remark}

\begin{document}

\title{A non-symmetric divide-and-conquer\\ recursive formula for the convolution\\ of polynomials and power series}
\date{November 2019}
\author[$\dagger$]{Thomas Baruchel}
\affil[$\dagger$]{\small Éducation nationale, France\authorcr\texttt{baruchel@riseup.net}}
\maketitle

\begin{abstract}
    \noindent Some changes in a recent convolution formula are performed here in order to clean it up by using more conventional notations and by making use of more referrenced and documented components (namely Sierpi\'nski's polynomials, the Thue-Morse sequence, the binomial modulo~2 transform and its inverse). Several variants are published here, by reading afterwards summed coefficients in another order; the last formula is then turned back from a summation to a new divide-and-conquer recursive formula.
\end{abstract}

\section{Introduction}
In a recently published paper, a new convolution formula was written down by tracking all terms along a recursion tree built from a variant of Karatsuba's well-known algorithm~\cite{baruchel}. While several variants of the formula were already given there, some unusual notations were heavily used in order to ``pack'' the required terms into a single summation. Despite the conciseness of these formulas, one could thus argue that they may be too complicated to stand as an inspiring starting point for ensuing researches.

Small changes in one of them can however lead to another more explicit variant, by noticing that three different arbitrary symbols are actually related to Sierpi\'nski's polynomials. The coefficients of these polynomials are those from the well-documented Sierpi\'nski triangle, and it may be expected that publishing a new simpler formula relying on such polynomials for something as significant as multiplicating two polynomials (or convoluting two power series) could have some benefits. In Sections~2 to Section~4, some elementary properties of Sierpi\'nski's polynomials are used in order to prove two new variants as Theorems~\ref{theorem1} and~\ref{theorem2}, one for polynomials and one for infinite power series.

Rewriting now Theorem~\ref{theorem2} by summing all embedded coefficients ``vertically'' rather than ``horizontally'' in Section~\ref{binomial} allows to take benefit of a little-studied transform called the ``binomial modulo~2 transform'' (and its inverse) in order to achieve a still more compact variant. The resulting formula in theorem~\ref{theorem4} being, like all these variants, a summation, it is then turned into to a new \emph{divide-and-conquer} recursive formula in section~\ref{recursive}.

This final recursive formula is said to be \emph{non-symmetric} because both convolved power series are not handled in the same way, one being handled by performing subtractions and the other by performing additions.

\section{Sierpi\'nski's polynomials}

The Sierpi\'nski triangle is best known as a graphical figure (see below); it is a fractal object built by adding at each iteration two new copies of the same whole object at its own bottom (and scaling down the whole figure in order to keep its original size). Thus iterating over its rows from top to bottom is possible: the $n$th~row (counting from the top) is always the same whatever the number of previous iterations is.
\begin{figure}[!ht]
    \begin{center}
    \includegraphics[scale=.25]{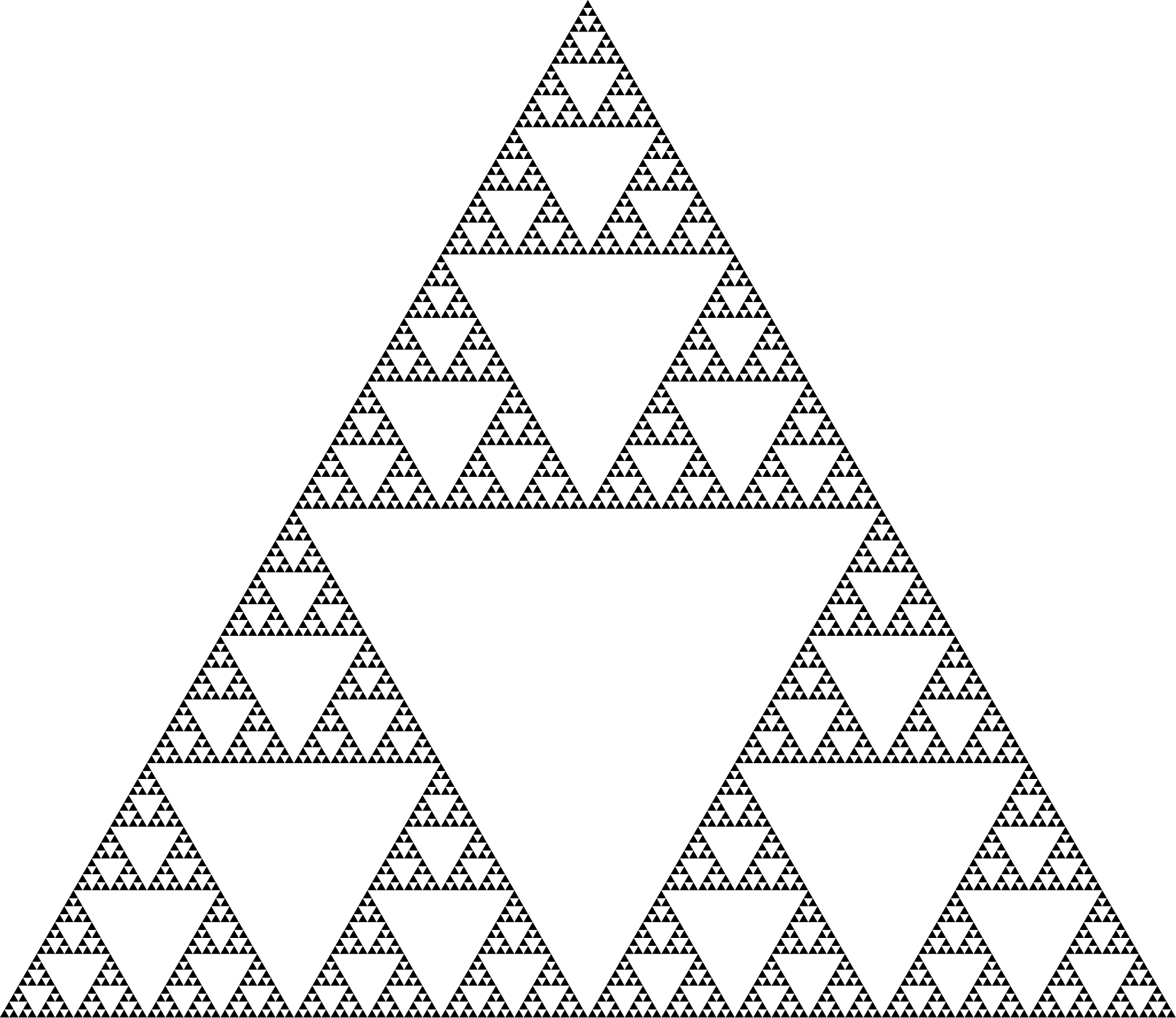}
\end{center}
\end{figure}
Iterating over the rows of the triangle and reading them as finite sequences of binary digits ($0$ for  ``white'' and $1$ for ``black'') gives another mathematical object which is also refferred to as the Sierpi\'nski triangle but now in some combinatorial context. Such coefficients are those from Pascal's triangle \emph{modulo}~$2$.

The sequential rows may also be read as polynomials by taking the previously described~$0$ and~$1$ as coefficients, resulting in the sequence~$S$ of Sierpi\'nski's polynomials defined among the comments of the sequence \texttt{A047999} in the \textit{On-Line Encyclopedia of Integer Sequences} \cite{oeis}. The initial polynomials are:
\[
    1,\, 1+x,\, 1+x^2,\, 1+x+x^2+x^3,\, \dots
\]

The power series expanding $(1-x)^{-1}S_k^{-1}$ will be referred to as~$\delta_k$ in the current paper (the symbol~$\delta$ being chosen because~$\delta_k$ happens to be the $k$th~main \emph{diagonal} in the Sierpi\'nski triangle).

All formulas below rely on the \emph{termwise} product of sequences (or power series), which will be denotated here as~$f\odot g$ (meaning that coefficients of the same rank are multiplicated with no convolution).

\begin{lemma}\label{lemma1}
    Let $m=(d_{n-1}\dots d_2 d_1 d_0)_2$ some nonnegative integer between~$0$ and~$2^n-1$ having the finite sequence~$d$ as the digits of its binary encoding; then
    \[
        S_m = \left(1+x\right)^{d_0}
              \left(1+x^2\right)^{d_1}
              \left(1+x^4\right)^{d_2}
              \,\dots\,
              \left(1+x^{2^{n-1}}\right)^{d_{n-1}}
              \,\textrm{.}
    \]
\end{lemma}
\begin{proof}
    This is obviously true for~$m=0$ and~$m=1$. Then, we refer to the building rule described as a comment of the sequence~\texttt{A047999}:
    \[
        \left\{
            \begin{array}{l}
            S_{2n+1}\left(x\right) = \left(x+1\right)S_n\left(x^2\right)
            \\[4pt]
            S_{2n}\left(x\right) = S_n\left(x^2\right)
            \end{array}
        \right.
    \]
    in order to prove by induction that if the lemma is true for any~$m$ smaller than some power of~$2$, it is also true for any value of~$m$ smaller than the following power of~$2$. The proof is straightforward since the binary encoding of~$2n$ is known to be the same as the one of~$n$ shifted to the left by one digit (thus performing $d_{j+1}\gets d_j$) while replacing~$x$ by~$x^2$ is the same as replacing
    \[
        \left(1+x^{2^j}\right)
        \quad\textrm{by}\quad
        \left(1+x^{2^{j+1}}\right)
\]
everywhere in the whole product above.
\end{proof}

\begin{lemma}\label{lemma2}
    Let $n$ be some power of~$2$ and $k$ some nonnegative integer smaller than~$n$; then:
    \[
        S_k\times S_{n-1-k} = S_{n-1}\,\textrm{.}
    \]
\end{lemma}
\begin{proof}
    The binary encoding of~$n-1$ is~$(111\dots 111)_2$ since~$n$ is a power of~$2$; thus $n-1-k$ and $k$ have complementary binary encodings. Thus, according to Lemma~\ref{lemma1}, $S_k$ and~$S_{n-1-k}$ have complementary factors in regards to the whole product defined in that lemma, which soon leads to the above statement.
\end{proof}

\section{The Thue-Morse sequence}

Let $\sigma$ be some specific encoding of the Thue-Morse sequence defined as the sequences \texttt{A106400} in the \textit{On-Line Encyclopedia of Integer Sequences} \cite{oeis}, namely:
\[
    1, -1, -1, 1, -1, 1, 1, -1, -1, 1, 1, -1, 1, -1, -1, 1, \dots
\]
Like the Sierpi\'nski triangle, the sequence may be built from a duplicating process: the $2^k$ initial coefficients are copied to their right with their sign being flipped in order to build the initial $2^{k+1}$ coefficients.

The notation $\bar{f}$ will be used here for representing the \emph{termwise} product of $\sigma$ and some polynomial~$f$ (or power series\footnote{The same notation could obviously also be used for a sequence of numbers if needed.}). We can thus write: $\bar{f}=\sigma\odot f$ with the symbol $\odot$ indicating the \emph{termwise} product of two objects.

Let also $\bar{S}$ be specifically the sequence of polynomials such that
    \[
        \bar{S}_m = \left(1-x\right)^{d_0}
              \left(1-x^2\right)^{d_1}
              \left(1-x^4\right)^{d_2}
              \,\dots\,
              \left(1-x^{2^{n-1}}\right)^{d_{n-1}}
    \]
    with $m=(d_{n-1}\dots d_2 d_1 d_0)_2$ some nonnegative integer between~$0$ and~$2^n-1$. This notation is very slightly different from the~$\bar{f}$ one since~$\bar{S}$ is a sequence of polynomials while~$f$ is a polynomial (or a power series), but it is easy to show that $\bar{S}_m=\overline{(S_m)}=\sigma\odot S_m$ by noticing that a coefficient in~$\bar{S}_m$ will be $-1$ if and only if an odd number of digits occurs in the binary encoding described above, which exactly matches the corresponding term in~$\sigma$ since $\sigma_n=(-1)^{\mathcal{H}_n}$ (with $\mathcal{H}_n$ being the Hamming weight of~$n$).

\begin{lemma}\label{lemma3}
    Let $n$ be some power of~$2$ and $k$ some nonnegative integer smaller than $n$; let also~$f$be some polynomials in the indeterminate~$x$; then:
    \[
        S_{n-1-k} \,x^k \,\odot\, \bar{S}_k \,f
        \,=\, \overline{S_{n-1-k} \,x^k \,\odot\, S_k\,\bar{f}}
        \,\textrm{.}
    \]
\end{lemma}
\begin{proof}
This is true for~$k=0$ and we want to prove by induction that when the identity is true for some~$k$ we can also write:
    \[
        S_{n-1-k-2^j} \,x^{k+2j} \,\odot\, \left(1-x^{2^j}\right)\bar{S}_k \,f
        \,=\, \overline{S_{n-1-k-2^j} \,x^{k+2^j}
        \,\odot\, \left(1+x^{2^j}\right)\overline{S_k\,\bar{f}}}
    \]
with~$2^j$ some power of~$2$ not already present in the binary expansion of~$k$.

The left-hand side of the previous equation means that we want to take consecutive blocks of~$2^{j+1}$ coefficients; subtract the initial~$2^j$ ones to the following~$2^j$ ones; and finally cancel the second half of such blocks (cancelling half of each block being performed by removing one more factor from the \emph{mask} $S_{n-1-k}$).

The right-hand side of the same equation performs the very same thing in another way: we flip the sign of coefficients in such a way that in consecutive blocks of~$2^{j+1}$ coefficients, the~$2^j$ initial ones will be flipped in an opposite manner than in the following~$2^j$ ones (this comes from the building rule of the Thue-Morse sequence); then we add (rather than subtract) the two parts; then we flip back the signs of the coefficients to their initial state.
\end{proof}

\section{New variants of the convolution formulas}

The general idea of a previous paper was to study a variant of Karatsuba's algorithm and ``flatten'' the recursion tree into a summation formula (see~\cite{baruchel}). Before going further, the latter is rewritten with more expressive notations in order to help manipulating it.

\begin{theorem}\label{theorem1}
    Let $f$ and $g$ two polynomials of degree~$n-1$ with $n$ some power of~$2$ in the same indeterminate~$x$; then
\[
    \begin{array}{lcl}
    f\times g
    &=&\displaystyle
    \sum_{k=0}^{n-1}
    \sigma_k \, S_{n-1-k}\left( S_{n-1-k}\, x^k \,\odot\, \bar{S}_k\, f \,\odot\, \bar{S}_k\, g \right)
    \\[16pt]
    &=&\displaystyle
    \sum_{k=0}^{n-1}
    \sigma_k \, S_{n-1-k}\left( S_{n-1-k}\, x^k \,\odot\, S_k\, \bar{f} \,\odot\, S_k\, \bar{g} \right) \,\textrm{.}
\end{array}
\]
\end{theorem}
\begin{proof}
We copy the the formula~(8) from~\cite{baruchel} as it is typeset in the original paper despite some differences in used notations; it will be translated to the current notations afterwards:
\[
    A\times B\, = \,
    \sum_{k=0}^{n-1}
    \bar{f}_k\left(\sigma_k \bar{f}_k X^k \odot \dot{f}_k A\odot\dot{f}_k B\right)
    \,\textrm{.}
\]
In the previous formula, $\dot{f}_k$ means exactly the same thing as~$\bar{S}_k$ in the current paper, while~$\bar{f}_k$ can be recognized here as~$S_{n-1-k}$ with the help of Lemma~\ref{lemma2}. Thus, the formula can now be translated as:
\[
    f\times g\, = \,
    \sum_{k=0}^{n-1}
    S_{n-1-k}\left(\sigma_k S_{n-1-k} x^k \odot \bar{S}_k f\odot \bar{S}_k g\right)
    \,\textrm{.}
\]
According to Lemma~\ref{lemma3}, $\bar{S}_k f$ and $\bar{S}_k g$ can be replaced above by $S_k\bar{f}$ and $S_k\bar{g}$ since all flipped signs will cancel themselves when evaluating the termwise product $S_k\bar{f}\odot S_k\bar{g}$, leading to the stated formula.
\end{proof}

\begin{theorem}\label{theorem2}
    Let $f$ and $g$ be two power series in the same indeterminate~$x$; then
\[
    \begin{array}{lcl}
    f\ast g
    &=&\displaystyle
    \sum_{k=0}^{\infty}
    \displaystyle\sigma_k \delta_k \left( x^k \delta_k \,\odot\, \bar{S}_k\, f \,\odot\, \bar{S}_k\, g \right)
    \\[16pt]
    &=&\displaystyle
    \sum_{k=0}^{\infty}
    \displaystyle\sigma_k \delta_k\left( x^k \delta_k \,\odot\, S_k\, \bar{f} \,\odot\, S_k\, \bar{g} \right)   \,\textrm{.}
\end{array}
\]
\end{theorem}
\begin{proof}
Formula~(9) from~\cite{baruchel} is now taken into account. We can not refer here to Lemma~\ref{lemma2} any longer for building some complementary polynomial because we do not work on a finite number~$n$ of terms, but the theory of generating function is useful for building the relevant power series, since
\[
    \frac{1}{1-x}\quad\textrm{expands to}\quad
              \left(1+x\right)
              \left(1+x^2\right)
              \left(1+x^4\right)
              \left(1+x^8\right)
              \,\dots\,
\]
where the required factors can easily be cancelled by a simple division: the previously defined~$\delta_k$ term matches the required infinite product.
\end{proof}

\section{The binomial modulo 2 inverse transform}\label{binomial}

In a comment of \texttt{A100735} in the \textit{On-Line Encyclopedia of Integer Sequences}~\cite{oeis} are defined the ``modulo~2 binomial transform'' and its inverse. In the current paper, the name of this transform is slightly changed to ``binomial modulo~2 transform''. Both transforms are defined as:
\[
    \begin{array}{lcl@{\qquad}l}
        B_n&=&\displaystyle\sum_{k=0}^n T_{n,k} A_k&\textit{(transform)} \\[14pt]
        A_n&=&\displaystyle\sum_{k=0}^n \sigma_{n-k} T_{n,k} B_k&\textit{(inverse transform)}
\end{array}
\]
for a given sequence~$A$, where $\sigma$ is the signed Thue-Morse sequence and $T_{n,k}$ is the relevant coefficient in the Sierpi\'nski triangle. The notations~$B=A'$ and~$A=B^*$ will be used from now on; the same notations will be applied to power series as well.

\begin{theorem}\label{theorem3}
    The power series $f$ being defined as~$f= a_0 + a_1 x + a_2 x^2 + \dots$, we define also $f_k=a_k + a_{k+1}x + a_{k+2}x^2 + \dots$, and~$f_k^*$ is the power series defined from the binomial modulo 2 inverse transform of the sequence~$a_k, a_{k+1}, a_{k+2},\dots$ (while $f'$ is the binomial modulo~2 transform of~$f$). Then for two power series~$f$ and~$g$ in the same indeterminate~$x$,
\[
    f\ast g =
    \sum_{k=0}^{\infty}
          \left(
              \,\overline{ \delta_k} \odot f_k^* \odot g_k^*
          \right)'
      \, x^k
      \,\textrm{.}
\]
\end{theorem}
\begin{proof}
    In Theorem~\ref{theorem2}, the whole parenthesis is turned into
    \[
        \sum_{j=0}^{\infty}
          \delta_j[k] \, f_j^*[k] \, g_j^*[k] \, x^{k+j}
    \]
where $f[k]$ is the coefficient of term of degree~$k$ in~$f$.

Then we gather all contributions of some degree~$m$ by noticing that some previously defined term of degree~$k+j$ is shifted to degree~$m$ (by performing the remaining multiplication in Theorem~\ref{theorem2}) if and only if~$m-(k+j)$ and $k$ have no common digit~$1$ in their binary encodings:
\[
    (f\ast g)[m] = \sum_{j=0}^{m}
        \sum_{k=0}^{m-j}
        T_{m-j,k}\,
        \sigma_k\,
          \delta_j[k] \, f_j^*[k] \, g_j^*[k]
\]
where $T$ is Sierpi\'nski triangle with $T_{n,k}=1$ if and only if $(n-k)\&k=0$, the symbol $\&$ being the bitwise \texttt{and} operator.

This identity quickly leads to the expected theorem.
\end{proof}


\begin{lemma}\label{slemma}
    The power series $f$ being defined as~$f= a_0 + a_1 x + a_2 x^2 + \dots$, we define also $f_k=a_k + a_{k+1}x + a_{k+2}x^2 + \dots$, and~$f_k^*$ is the power series defined from the binomial modulo 2 inverse transform of the sequence~$a_k, a_{k+1}, a_{k+2},\dots$ (with of course $f^* = f_0^*$). Then,
    \[
        \left( \delta_k \odot f_k^* \right) x^k
        = \delta_k \, x^k \odot f^* S_k
        \,\textrm{.}
    \]
\end{lemma}
\begin{proof}
    We prove this by induction; the statement is obviously true for~$k=0$; then we assume it is true for some nonnegative integer~$k$, and we show that it is still true for~$k+n$ with $n$ being some power of~$2$ greater than~$k$. Let $g=f_k$ and $g_n=f_{k+n}$ in order to focus on a power of~2; we want to study $\left(\delta_{k+n}\odot g_n^*\right) x^{k+n}$ in order to match the left-hand side above.

    Because of the self-similarity in the Sierpi\'nski triangle, we can notice that $\left(1+x^n\right)g^*$ and~$x^n g_n^*$ share the same coefficient of degree~$m$ if $m \,\textrm{mod}\,\, 2n \geqslant n$. For the same reasons, $\delta_k$ and $x^n \delta_{k+n}$ share the same coefficient of degree~$m$ if $m \,\textrm{mod}\,\, 2n \geqslant n$ (all other coefficients of $\delta_{k+n} x^n$ are null). Thus,
    \[
        \begin{array}{lcl}
            \left(\delta_{k+n}\odot f_{k+n}^*\right) x^{k+n}
            &=& \left(\delta_{k+n}\odot g_n^*\right) x^{k+n} \\[4pt]
            &=& \left(x^n \delta_{k+n} \odot x^n g_n^*\right) x^{k} \\[4pt]
            &=& \left( x^n\delta_{k+n}\odot\left(1+x^n\right) g^* \right)x^k \\[4pt]
            &=& x^{k+n}\delta_{k+n}\odot\left(1+x^n\right) x^k f_k^*  \\[4pt]
        \end{array}
    \]
    where, again, the (rather restrictive) mask $x^{k+n}\delta_{k+n}$ allows to replace $x^k f_k^*$ with $f^* S_k$ (by using the initial assumption) \emph{because all coefficients to be added and kept in the multiplication by $\left(1+x^n\right)$ where also taken into account by the initial (less restrictive mask)}. Of course $\left(1+x^n\right)f^* S_k = f^* S_{k+n}$.
\end{proof}

\begin{lemma}\label{slemma2}
    Let $f$ be some power series in the indeterminate~$x$. Then,
    \[
        \left( \delta_k \odot f \right)' x^k
        = \left( \left( \delta_k \odot f\right) x^k \right)' S_k
        \,\textrm{.}
    \]
\end{lemma}
\begin{proof}
    We prove this by induction; the statement is obviously true for~$k=0$; then we assume it is true for some nonnegative integer~$k$, and we show that it is still true for~$k+n$ with $n$ being some power of~$2$ greater than~$k$.

    Because of the self-similarity in the Sierpi\'nski triangle, we can notice that all coefficients of $\left(\delta_{k+n}\odot f\right)'$ in the left-hand part above can be found among the coefficients of~$\left(\delta_{k}\odot f\right)'$ according to the following rule:
    \[
    \left\{
    \begin{array}{l@{\qquad}l}
        \left(\delta_{k+n}\odot f\right)'[m]
            = \left(\delta_{k}\odot f\right)'[m] & \text{ if $m \,\textrm{mod}\,\, 2n < n$; }\\[6pt]
        \left(\delta_{k+n}\odot f\right)'[m]
            = \left(\delta_{k}\odot f\right)'[m-n] & \text{ if $m \,\textrm{mod}\,\, 2n \geqslant n$; }
            \end{array}
    \right.
    \]
which can also be written as $\left(\delta_{k+n}\odot f\right)' = \left(\delta_n\odot \left(\delta_{k}\odot f\right)'\right)\left(1+x^n\right)$.

 The same idea is separately followed for the right-hand part of the identity to be proved: $\left( \left( \delta_{k+n} \odot f\right) x^{k+n} \right)' S_{k+n}$ is also made of coefficients from $\left( \left( \delta_k \odot f\right) x^k \right)' S_k$ according to the very same same rule because of the self-similarity property of the Sierpi\'nski triangle again. Thus we see that both sides of the statement match exactly.
\end{proof}

\begin{theorem}\label{theorem4}
    Let $f$ and $g$ be two power series in the same indeterminate~$x$. Then,
    \[
    f\ast g =
    \sum_{k=0}^{\infty}
          \left(
              \,
              \overline{ \delta_k}\, x^k \,\odot\, f^* S_k \,\odot\, g^* S_k
          \right)'
      \, S_k
      \,\textrm{.}
    \]
\end{theorem}

\begin{proof}
    Theorem~\ref{theorem3} is rewritten with the help of Lemma~\ref{slemma}:
    \[
    f\ast g =
    \sum_{k=0}^{\infty}
          \left(
              \,\displaystyle\frac
               {
              \overline{ \delta_k}\, x^k \,\odot\, f^* S_k \,\odot\, g^* S_k
               }
               {x^k}
          \right)'
      \, x^k
    \]
which quickly leads to the expected statement with the help of Lemma~\ref{slemma2}.
\end{proof}

\section{A recursive formula for the convolution}\label{recursive}

Summations involving Sierpi\'nski's polynomials may be easy to convert into a recursive formula because ``jumping'' from $S_k$ to $S_{k+n}$ (with $n$ being a power of~2 greater than~$k$) is merely achieved by using the relation~$S_{k+n}=(1+x^n)S_k$.

Furthermore, the parenthesis in Theorem~\ref{theorem4} contains terms which still follow the previously studied \emph{interleaved splitting scheme} (see Section~3 in~\cite{baruchel}), meaning that they can be computed from another ones by splitting them into two increasingly-sparse terms. An exact definition of the scheme used below is:
\[
       \left\{
       \begin{array}{lcl}
       u_{\text{\tiny low}}^{\text{\tiny $(n)$}}
         &=& u\odot \delta_n \\[5pt]
       u_{\text{\tiny high}}^{\text{\tiny $(n)$}}
         &=& u\odot x^n \delta_n
       \end{array}\right.
       \qquad\text{with $n$ being some power of~$2$.}
\]

For traversing the recursion tree, an \emph{ad hoc} operator will be used here in order to ``pack'' the summation from Theorem~\ref{theorem4} into the following formula:
\begin{equation}\label{recformula}
f\ast g \,=\, \left(\,\overline{f^*}\odot g^*\right)'
        \,+\, \overline{f^*}
              \underset{\text{\tiny $(1)$}}{\diamond} g^*
\end{equation}
this parametrized operator symbol $\underset{\text{\tiny $(n)$}}{\diamond}$ being recursively defined as:
\begin{equation}\label{recoperator}
\begin{array}{lcl}
    u \underset{\smash{\text{\tiny $(n)$}}}{\diamond} v
     &=& \left(1+x^n\right)
       \left( x^n u_{\text{\tiny low}}^{\text{\tiny $(n)$}} - u_{\text{\tiny high}}^{\text{\tiny $(n)$}}
               \, \odot \,
              x^n v_{\text{\tiny low}}^{\text{\tiny $(n)$}} + v_{\text{\tiny high}}^{\text{\tiny $(n)$}}
       \right)' \\[4pt]
     &&+\, u \underset{\smash{\text{\tiny $(2n)$}}}{\diamond} v \\[4pt]
     &&+\, \left(1+x^n\right)
           \left(
        x^n u_{\text{\tiny low}}^{\text{\tiny $(n)$}} - u_{\text{\tiny high}}^{\text{\tiny $(n)$}} 
            \, \underset{\smash{\text{\tiny $(2n)$}}}{\diamond} \,
              x^n v_{\text{\tiny low}}^{\text{\tiny $(n)$}} + v_{\text{\tiny high}}^{\text{\tiny $(n)$}}
              \right)
\end{array}
\end{equation}
(where the digits in the binary encoding of the variable~$k$ occuring in Theorem~\ref{theorem4} are read from right to left by either choosing a digit~0 or a digit~1, explaining the two recursive calls).

Obiously the previously-defined operator is non-commutative. Building this formula as a non-symmetric one allows to track and preserve the signs from the~$\overline{\delta_k}$ term without using an external mask (which is the case in Theorem~\ref{theorem4}). Because of the properties of the Thue-Morse sequence, subtractions actually embed ``hidden'' additions (since exactly one of both terms occuring in each subtractions previously had its sign flipped).

\section{Conclusion}

The new theorems~\ref{theorem1} and~\ref{theorem2} are equivalent to their previous versions in~\cite{baruchel}, but they now make use of more documented notations. They are therefore claimed to be of a more general interest for later researches. The binomial modulo~2 transform is used a couple of times in the \textit{On-Line Encyclopedia of Integer Sequences}, but it does not seem to have been actually studied in published papers; thus theorems and lemmas proved in Section~\ref{binomial} and~\ref{recursive} may also provide a basis for further investigating the properties of this transform. The conciseness of formulas in Theorems~\ref{theorem3} and~\ref{theorem4} show that ``transposing'' the initial ones suit more the approach being investigated since the previous paper.

Formula~(\ref{recoperator}) in Section~\ref{recursive} ``packs'' the whole computation into a rather simple statement with no need for symbols related to the Sierpi\'nski triangle any longer, allowing to focus henceforth on the binomial modulo~2 transform only. Directly implementing this formula as code or pseudo-code was not considered at this point, because it would certainly not be very efficient as it is; it should be noticed however that while the definition of the binomial modulo 2 transform in Section~\ref{binomial} obviously is in~$O(n^2)$, much better algorithms for computing the whole transform in~$O(n\log n)$ are achievable. Once initially computed, the transform could probably be propagated and partially updated during the recursion process if attempting to efficiently implement the formula.

\vspace{32pt}
\noindent{\small \textbf{Conflict of Interest:} The authors declare that they have no conflict of interest.}

\vspace{8pt}
\noindent{\small The current article is accessible on \texttt{http://export.arxiv.org/pdf/1912.00452}\hspace{2pt}.}

\end{document}